\documentclass[11pt,reqno]{amsart}

\usepackage{hyperref}
\usepackage{graphicx}
\usepackage{color}
\usepackage{amsfonts,amssymb}
\usepackage{bbm} 
\usepackage{psfrag}
\usepackage{mathrsfs}
\usepackage{a4wide}

\usepackage{ifthen}
\usepackage{psfrag}
\newcommand{\frag}[3][f]{
\ifthenelse{\equal{#1}{f}}
{
    \psfrag{#2}{\footnotesize#3}
}{
    \ifthenelse{\equal{#1}{s}}
    {
        \psfrag{#2}{\small#3}
    }{
        \ifthenelse{\equal{#1}{l}}
        {
            \psfrag{#2}{\large#3}
        }{
            \ifthenelse{\equal{#1}{ss}}
            {
                \psfrag{#2}{\scriptsize#3}
            }{
                \ifthenelse{\equal{#1}{n}}
                {
                    \psfrag{#2}{#3}
                }{
                }
            }
        }
    }
}
}
\newcommand{\hess}{\mathrm{Hess}}

\newtheorem{neu}{}[section]
\newtheorem{Cor}[neu]{Corollary}
\newtheorem*{Cor*}{Corollary}
\newtheorem{Thm}[neu]{Theorem}
\newtheorem*{Thm*}{Theorem}
\newtheorem{Prop}[neu]{Proposition}
\newtheorem*{Prop*}{Proposition}
\newtheorem{Lemma}[neu]{Lemma}
\theoremstyle{definition}
\newtheorem*{Rmk*}{Remark}
\newtheorem{Rmk}[neu]{Remark}

\newtheorem*{Ex*}{Example}

\newtheorem*{Qu*}{Question}

\newcommand{\N}{\mathbb{N}}

\newcommand{\Z}{\mathbb{Z}}
\newcommand{\R}{\mathbb{R}}

\newcommand{\pf}{\longrightarrow}

\newcommand{\CZ}{\mu_{\mathrm{CZ}}}

\newcommand{\Morse}{\mu_{\mathrm{Morse}}}

\newcommand{\om}{\omega}

\newcommand{\Poincare}{Poincar\'{e} }

\newcommand{\B}{\mathcal{B}}
\newcommand{\CCC}{\mathcal{C}}

\renewcommand{\L}{\mathscr{L}}

\newcommand{\Ham}{\mathrm{Ham}}

\newcommand{\Crit}{\mathrm{Crit}}

\newcommand{\supp}{\mathrm{supp}}

\newcommand{\beq}{\begin{equation}}
\newcommand{\beqn}{\begin{equation}\nonumber}
\newcommand{\eeq}{\end{equation}}

\newcommand{\bea}{\begin{equation}\begin{aligned}}
\newcommand{\bean}{\begin{equation}\begin{aligned}\nonumber}
\newcommand{\eea}{\end{aligned}\end{equation}}

\numberwithin{equation}{section}

\definecolor{green}{rgb}{0,.7,0}
\definecolor{blue}{rgb}{0,0,1}
\definecolor{red}{rgb}{1,0,0}

\newcommand{\p}{\partial}

\begin{document}
\title{Periodic bounce orbits of prescribed energy}
\author{Peter Albers}
\author{Marco Mazzucchelli}
\address{
    Peter Albers\\
    Department of Mathematics\\
    Purdue University}
\email{palbers@math.purdue.edu}
\address{
    Marco Mazzucchelli\\
    Max Planck Institute for Mathematics in the Sciences\\
    Leipzig}
\email{mazzucch@mis.mpg.de}
\date{May 24, 2010}
\begin{abstract}
We prove the existence of periodic bounce orbits of prescribed energy on an open bounded domain in $\R^N$. We derive explicit bounds on the period and the number of bounce points. 
\end{abstract}
\maketitle

\section{Introduction}

Throughout this article we fix an open, bounded domain $\Omega\subset\R^N$ with smooth boundary and a smooth function $V\in C^\infty(\overline{\Omega}).$\footnote{We expect that all results remain true if $C^\infty$ is replaced by $C^2$.} We study periodic bounce orbits of the Lagrangian system given by
\bea
L:T\overline\Omega=\overline{\Omega}\times\R^N&\to\R\\
(q,v)&\mapsto\tfrac12|v|^2-V(q)
\eea
that is, continuous and piecewise smooth maps $\gamma:\R/\tau\Z\to\overline{\Omega}$, $\tau>0$,  satisfying the following. There exists a (possibly empty) finite subset $\B\subset[0,\tau]$ such that
\begin{enumerate}
\item $\gamma$ solves the Euler-Lagrange equation
\beq\label{eqn:2nd_order_ODE}
\gamma''(t)+\nabla V(\gamma(t))=0\quad\forall t\not\in\B\;
\eeq
\item for each $t\in\B$ we have $\gamma(t)\in\p\Omega$, the left resp.~right derivatives 
\beq
\gamma'(t^\pm):=\lim_{s\to t^\pm}\gamma'(s)
\eeq
exist and $\gamma$ satisfies the law of reflection
\bea\label{e:law_of_reflection}
\big\langle \gamma'(t^+), \nu(\gamma(t)) \big\rangle&=-\big\langle \gamma'(t^-), \nu(\gamma(t))\big\rangle\neq0\;,\\
\gamma'(t^+)-\big\langle \gamma'(t^+), \nu(\gamma(t))\big \rangle\cdot \nu(\gamma(t))&=\gamma'(t^-)-\big\langle \gamma'(t^-), \nu(\gamma(t))\big\rangle \cdot\nu(\gamma(t))\;,
\eea
where $\nu$ is the outer normal to $\p\Omega$.
\end{enumerate}

\begin{Rmk}$ $
\begin{itemize}
\item The times $t\in\B$ are called bounce times and $\gamma(t)$ bounce points. In case $V$ is a constant function bounce orbits are billiard trajectories, see \cite{Kozlov_Treshchev_Billiards,Tabachnikov_Geometry_and_billiards} for more details on billiards. 
\item A periodic bounce orbit with $\B=\emptyset$ is a smooth periodic solution of \eqref{eqn:2nd_order_ODE}.
\item For a periodic bounce orbit $\gamma$ the energy
\beq
E(\gamma):=\tfrac12|\gamma'(t)|^2+V(\gamma(t))
\eeq
is an integral of motion, namely it is independent of $t\not\in\B$.
\end{itemize}
\end{Rmk}

\begin{Thm}\label{thm:main}
For all $E>\max_{\overline{\Omega}}V$ there exists a periodic bounce orbit $\gamma:\R/\tau\Z\pf\overline{\Omega}$ with energy $E(\gamma)=E$, at most $\dim\Omega+1$ bounce points, and period bounded as follows
\beq\label{eqn:period_inequality}
\tau\leq C\,\mathrm{diam}(\Omega)\,\frac{\big(E-{\textstyle\min_{\overline\Omega}}V\big)^{5/2}}{(E-\textstyle\max_{\overline\Omega}V)^3}\;,
\eeq
where $C$ is a constant independent of $\Omega$, $V$ and $E$ $($see Propositions~\ref{prop:uniform_contact_type_estimate} and~\ref{prop:approx_sequence} for an explicit estimate for $\tau$$)$.
\end{Thm}

\begin{Rmk}
In dimension two the bound on the number of bounce points in sharp in general. In fact, already for billiard trajectories there are domains $\Omega\subset\R^2$ where every billiard trajectory has at least three bounce points, see for instance \cite[Figure 6.6]{Tabachnikov_Geometry_and_billiards}. It is conceivable that the bound on the number of bounce points is also sharp in higher dimension.
\end{Rmk}

\begin{Cor}\label{cor:main}
If in Theorem~\ref{thm:main} we further require
\beq
E(\gamma)> 
\textstyle\max_{\overline\Omega}V
+
\tfrac12\mathrm{diam}(\overline\Omega)\textstyle\max_{\overline\Omega}|\nabla V|
\eeq
then the periodic bounce orbit $\gamma$ has at least one bounce point.
\end{Cor}

The proofs of Theorem \ref{thm:main} and Corollary \ref{cor:main} are carried out at the end of Section~\ref{sec:proof_main}. Inequality \eqref{eqn:period_inequality} confirms the physical intuition that there exist orbits whose period decreases as the energy increases. Moreover, asymptotically the minimal period decreases at least as fast as the inverse of the square root of the energy.

\begin{Rmk}
In their influential work \cite{Benci_Giannoni_Periodic_bounce_trajectories_with_a_low_number_of_bounce_points} Benci-Giannoni prove existence of periodic bounce orbits of prescribed period and with at most $\dim\Omega+1$ bounce points. This is achieved by studying the classical fixed-time action functional of an approximating smooth Lagrangian system. In this article we replace this by the free-time action functional. Therefore, we detect periodic orbits of prescribed energy rather than period. 

A new difficulty in the approximation scheme is to obtain bounds on the periods for approximate solutions independent of the approximation parameter. This is necessary to pass to the limit. To achieve this we employ techniques from symplectic geometry as opposed to the variational techniques used by Benci-Giannoni. This also enables us to give explicit bounds on the period of the periodic bounce orbits in the limit.

We point out that in the case of a constant potential $V$, say $V\equiv 0$, the result by Benci-Giannoni  and the statement of Theorem \ref{thm:main} reduce to the mere existence of only one periodic billiard trajectory. In fact, if $V\equiv0$, given any $T$-periodic billiard trajectory  $\gamma$ of energy $E$, the reparametrized curve $\gamma(\frac{\cdot}{\tau})$ is a $\tau T$-periodic billiard trajectory of energy $\tau^{-2}E$.
\end{Rmk}

\begin{Rmk}
Finally, we want to mention two natural generalizations of the set-up considered here. Both seem nontrivial to us, and we will treat them  further in future research.

The first generalization is to allow general Riemannian metrics. The approximation scheme can be formulated entirely in Riemannian terms and we are optimistic that is carries over. The same applies to the symplectic topology part. Nevertheless, it is harder to ensure the existence of bouncing points for a sequence of approximating solutions. Indeed, if the Riemannian metric allows a closed geodesic in $\Omega$ and the potential $V$ vanishes along such a geodesic then for any energy this closed geodesic (suitably reparametrized) gives a periodic orbit with no bounce points. 

Another possible generalization is to add a magnetic field, i.e.\ ``twisting'' the symplectic structure on $T^*\overline{\Omega}$ by adding to the canonical symplectic form a closed 2-form $\sigma$ defined on the base $\overline{\Omega}$. If $\sigma$ is non-exact then it seems impossible to generalize the methods employed here. First or all, there is no Lagrangian formulation of the problem, in particular, there is no approximation scheme. Second, the approximating energy hypersurfaces in the Hamiltonian formulation may cease to be of contact type, in particular, might be without periodic orbits. Moreover, there is no period-action inequality. These two problems disappear if the magnetic field $\sigma$ is exact: indeed, in this case case, twisting the symplectic form on $T^*\overline{\Omega}$ amounts to adding a primitive of $\sigma$ to the Hamiltonian while  keeping the canonical symplectic structure. Then,  there is a Lagrangian formulation and the energy hypersurfaces are of contact type for sufficiently large energy. Nevertheless, the statement of the approximation scheme doesn't readily generalize since near the boundary the magnetic field interacts with the perturbation potential. Also, from a physical point of view one might expect to see ``creeping'' orbits, that is, orbits which after bouncing are very soon forced back towards the boundary by the magnetic field. Thus, effective bounds on the number of bounce points might be hard to obtain. 
\end{Rmk}

\subsection*{Organization of the article}
In Section \ref{sec:approximation_scheme} we define the approximation scheme for the free-time action functional and prove that a sequence of approximating  solutions converge to periodic bounce orbits of prescribed energy provided their Morse index is uniformly bounded. In Section \ref{sec:proof_main} we study the Hamiltonian formulation of the approximation scheme and prove existence of solutions using techniques from symplectic geometry. Moreover, we derive effective bounds on the period and Morse index. Combining this with the results from Section~\ref{sec:approximation_scheme} leads to a proof of Theorem~\ref{thm:main}.

\subsection*{Acknowledgments}
This article was written during visits of the authors at the Institute for Advanced Study, Princeton. The authors thank the Institute for Advanced Study for its stimulating working atmosphere. The second author has been supported by a postdoctoral fellowship granted by the Max Planck Institute for Mathematics in the Sciences (Leipzig, Germany). Both authors thank Felix Schlenk for helpful remarks.

This material is based upon work supported by the National Science Foundation under agreement No.~DMS-0635607 and DMS-0903856. Any opinions, findings and conclusions or recommendations expressed in this material are those of the authors and do not necessarily reflect the views of the National Science Foundation.

\section{The approximation scheme}\label{sec:approximation_scheme}

Our proof of Theorem \ref{thm:main} makes it necessary to modify the beautiful approximation scheme due to Benci-Gianonni \cite{Benci_Giannoni_Periodic_bounce_trajectories_with_a_low_number_of_bounce_points} by replacing the fixed-time action functional by the free-time action functional.

We recall that $\Omega\subset\R^N$ is an open, bounded domain with smooth boundary and $V\in C^\infty(\overline{\Omega})$. We fix $d_0\in(0,\tfrac12)$ sufficiently small, in particular, such that the distance function $\mathrm{dist}_{\p\Omega}(q)=\min\bigl\{ |q-q'|\ \bigl|\ q'\in\partial\Omega \bigr\}$  is smooth at all points $q\in\Omega$ with $\mathrm{dist}_{\p\Omega}(q)\leq 2d_0$. Let $k:[0,\infty)\to[0,2d_0]$ be a smooth function such that $0\leq k'\leq1$, $k(x)=x$ if $x\leq d_0$ and $k(x)=\mathrm{const}$ if $x\geq2d_0$. Then, we define a function $h\in C^\infty(\overline{\Omega})$ by 
\beq
h(q):=k(\mathrm{dist}_{\p\Omega}(q)).
\eeq
Notice that $h$ satisfying the following.
\begin{itemize}
\item $h(q)=\text{dist}_{\p\Omega}(q)$ for all $q\in \overline{\Omega} $ with $\text{dist}_{\p\Omega}(q)\leq d_0$,
\item $h(q)>d_0$ if $\text{dist}_{\p\Omega}(q)> d_0$,
\item $0\leq h\leq 1$ and $h(q)=\mathrm{const}$ if $\text{dist}_{\p\Omega}(q)\geq2d_0$,
\item $|\nabla h|\leq1$.
\end{itemize}
Finally we define a function $U\in C^\infty(\Omega)$ by
\beq\label{e:definition_of_U}
U(q):=\frac{1}{h^2(q)}\;.
\eeq
Thus, $U$ is a positive function that grows like $(\text{dist}_{\p\Omega})^{-2}$ near $\partial\Omega$ and is constant in the region $\{\mathrm{dist}_{\p\Omega}(q)\geq2d_0\}$, see Figure~\ref{fig:potential_U}.

\frag[n]{H}{$U$}
\frag[n]{w}{$\partial\Omega$}
\frag[n]{O}{$\Omega$}
\frag[n]{0}{$0$}
\frag[n]{z}{$h$}
\begin{figure}[h]
\includegraphics[scale=1]{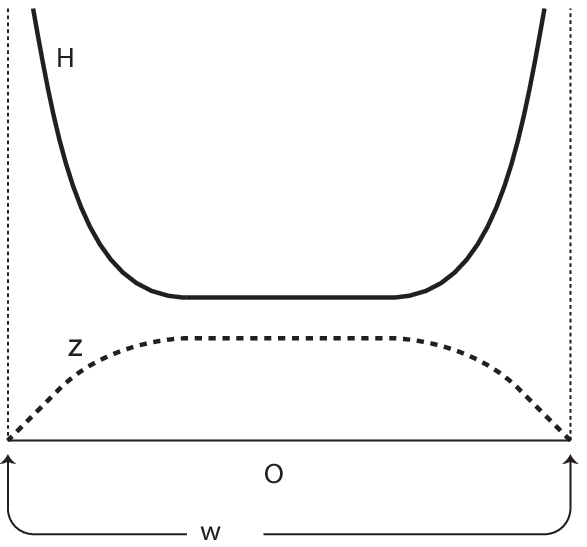}
\caption{}\label{fig:potential_U}
\end{figure}

For $\epsilon>0$, we introduce the modified Lagrangian 
\bea\label{e:Lagrangian_epsilon}
L_\epsilon:T\Omega=\Omega\times\R^N&\to\R\\
(q,v)&\mapsto\tfrac12|v|^2-V(q)-\epsilon U(q)\;.
\eea
For each energy value $E\in\R$ the free-time action functional $\L^E_\epsilon:H^1(\R/\Z,\overline{\Omega})\times\R_{>0}\to\R$ is given by
\bea
\L^E_\epsilon(\Gamma,\tau):=\tau\int_0^1 \Big[L_\epsilon \big(\Gamma(t),\tfrac{1}{\tau}\Gamma'(t) \big)+E\Big]dt\;
=
\int_0^\tau \Bigl[L_\epsilon(\gamma,\gamma'\bigl)+E \Bigr]dt,
\eea
where $\gamma(t):=\Gamma(\tfrac t\tau)$. The differential of $\L^E_\epsilon$ is given by
\bea\label{e:diff_action}
d\L^E_\epsilon(\Gamma,\tau)(\Psi,\sigma)
= & \,
\tau
\int_0^1
\Bigl[
\partial_v L_\epsilon\bigl(\Gamma,\tfrac1\tau \Gamma'\bigr) \tfrac1\tau\Psi'
+
\partial_q L_\epsilon\bigl(\Gamma,\tfrac1\tau \Gamma'\bigr) \Psi
\Bigr]
dt\\
&
+
\sigma
\int_0^1 
\Bigl[
E-\partial_v L_\epsilon\bigl(\Gamma,\tfrac1\tau \Gamma'\bigr) \tfrac1\tau\Gamma'
-
L_\epsilon\bigl(\Gamma,\tfrac1\tau \Gamma'\bigr)\Bigr]
dt\\
=&\,
\tau
\int_0^1
\Bigl[
\tau^{-2}
\langle
\Gamma',\Psi'
\rangle
-
\langle
\nabla V(\Gamma)+\epsilon\nabla U(\Gamma), \Psi
\rangle
\Bigr]
dt\\
&
+
\sigma
\int_0^1 
\Bigl[
E
-
(2\tau^2)^{-1} |\Gamma'|^2 + V(\Gamma) + \epsilon U(\Gamma)
\Bigr]
dt\\
=&\,
\int_0^\tau
\Bigl[
\langle
\gamma',\psi'
\rangle
-
\langle
\nabla V(\gamma)+\epsilon\nabla U(\gamma), \psi
\rangle
\Bigr]
dt\\
&
+
\frac \sigma\tau
\int_0^1 
\Bigl[
E
-
\tfrac12 |\gamma'|^2 + V(\gamma) + \epsilon U(\gamma)
\Bigr]
dt,
\eea
where $\psi(t):=\Psi(\tfrac t\tau)$. Therefore $(\Gamma,\tau)$ is a critical point of $\L^E_\epsilon$ 
if and only if the corresponding $\tau$-periodic  curve $\gamma$ is a solution of the Euler-Lagrange system
\beq\label{eqn:Euler_Lagrange}
\gamma''+\nabla V(\gamma)+\epsilon\nabla U(\gamma)=0
\eeq
with energy 
\beq\label{eqn:energy_of_gamma}
E_\epsilon(\gamma):=\tfrac12|\gamma'(t)|^2+V(\gamma(t))+\epsilon U(\gamma(t))=E\;.
\eeq
We prove the analogue of \cite[Proposition~2.3]{Benci_Giannoni_Periodic_bounce_trajectories_with_a_low_number_of_bounce_points} for the free-time action functional.
\begin{Prop}\label{prop:Benci_improved}
Let $K>0$ and $T_2>T_1>0$. For each $\epsilon>0$, let $(\Gamma_\epsilon,\tau_\epsilon)$ be a critical point of $\L^{E_\epsilon}_\epsilon$ with $T_1\leq \tau_\epsilon\leq T_2$ and $E_\epsilon\leq K$. Then, up to a subsequence, $(\Gamma_\epsilon,\tau_\epsilon)$  converges to $(\Gamma,\tau)$ in $H^1(S^1,\overline{\Omega})\times\R_{>0}$ as $\epsilon\to0$. Moreover, if we define the curve $\gamma(t):=\Gamma(\tfrac{t}{\tau})$, there exists a finite Borel measure $\mu$  on $\CCC=\{t\in\R/\tau\Z\,|\,\gamma(t)\in\partial\Omega\}$ such that
\begin{itemize}
\item[(i)] $\displaystyle\int_0^\tau  \Bigl[ \langle \gamma',\psi' \rangle - \langle \nabla V(\gamma),\psi \rangle \Bigr] dt = \int_{\CCC} \langle \nu(\gamma),\psi \rangle\, d\mu $ for all $\psi\in H^1(\R/\tau\Z;\R^N)$,
\item[(ii)] $\gamma$ is a smooth solution of the Euler-Lagrange system of $L$ outside $\supp(\mu)$, with energy $E(\gamma)=\lim_{\epsilon\to0} E_\epsilon$,
\item[(iii)] $\gamma$ has left and right derivatives that are left and right continuous on $\R/\tau\Z$ respectively. Moreover, $\gamma$ satisfies the law of  reflection~\eqref{e:law_of_reflection} at each time $t$ which is an isolated point of $\supp(\mu)$.
\end{itemize}
In particular, if $\supp(\mu)$ is a  finite set then $\gamma$ is a periodic bounce orbit of the Lagrangian system given by $L$ and $\B:=\supp(\mu)$ is its set of  bouncing times.
\end{Prop}

\begin{proof}
Since the sequences $\{\tau_\epsilon\}$ and $\{E_\epsilon\}$ are bounded, up to a subsequence for $\epsilon\to0$, we have $\tau_\epsilon\to\tau$ and $E_\epsilon\to E$ with $T_1\leq\tau\leq T_2$ and $E\leq K$. We show that up to further passing to a subsequence, $\Gamma_\epsilon$ also converges in $H^1$.

Let $\gamma_\epsilon(t)=\Gamma_\epsilon(\tfrac{t}{\tau_\epsilon})$ be the periodic orbit corresponding to $(\Gamma_\epsilon,\tau_\epsilon)$. By equations \eqref{e:diff_action} and $\eqref{eqn:energy_of_gamma}$ we know that the energy $E_\epsilon(\gamma_\epsilon)$ is equal to $E_\epsilon$, and therefore
\beq
\label{e:energy_epsilon}
(2\tau_\epsilon)^{-1} |\Gamma_\epsilon'|^{2}
+
V(\Gamma_\epsilon)
+
\epsilon
U(\Gamma_\epsilon)
\equiv
E_\epsilon
\;.
\eeq
Moreover, $\gamma_\epsilon$ is a solution of the Euler-Lagrange equation \eqref{eqn:Euler_Lagrange} associated to $L_\epsilon$, which can be written in terms of $(\Gamma_\epsilon,\tau_\epsilon)$ as
\beq
\label{e:EulerLagrange_epsilon}
\tau_\epsilon^{-2} 
\Gamma_\epsilon''
+
\nabla V(\Gamma_\epsilon)
+
\epsilon\nabla U(\Gamma_\epsilon)
=0.
\eeq
In particular, for each $(\Psi,\sigma)\in H^1(S^1;\R^N)\times\R$ we have $d\L^{E_\epsilon}_\epsilon(\Gamma_\epsilon,\tau_\epsilon)(\Psi,\sigma)=0$ and choosing $\sigma=0$ in equation~\eqref{e:diff_action} we get
\bea\label{e:weakEulerLagrange_epsilon}
\int_0^1
\Bigl[
\tau_\epsilon^{-2}
\langle
\Gamma_\epsilon',\Psi'
\rangle\,
-
\langle
\nabla V(\Gamma_\epsilon),\Psi
\rangle
\Bigr]
dt
=
\int_0^1
\langle
\epsilon\nabla U(\Gamma_\epsilon)
,\Psi
\rangle\,
dt,\quad
\forall \Psi\in H^1(S^1;\R^N).
\eea
We fix $\Psi=\Psi_\epsilon=-\nabla h(\Gamma_\epsilon)$. By equation~\eqref{e:energy_epsilon}, $\Gamma_\epsilon'$ is uniformly bounded in $L^\infty$, and so is $\Psi_\epsilon'$. Hence, with our choice of $\Psi$ the first two summands on the left hand side of~\eqref{e:weakEulerLagrange_epsilon} are uniformly bounded in $\epsilon$, and thus, so must be the third summand, i.e.
\beq\label{e:bg_bound1}
\int_0^1
\langle
\epsilon\nabla U(\Gamma_\epsilon)
,\Psi_\epsilon
\rangle\;
dt
=
\int_0^1
\frac{2\epsilon}{ h^3(\Gamma_\epsilon)}|
\nabla h(\Gamma_\epsilon)|^2
\;
dt
\leq C.
\eeq
Let $\Omega'\subset\Omega$ be the compact neighborhood of $\partial\Omega$ given by
\beq
\Omega'=\{ q\in\Omega\,|\,h(q)\leq d_0 \},
\eeq
where $d_0$ is the positive constant that enters the definition of the function $h$. Notice that on $\Omega'$ we have $h=\mathrm{dist}_{\partial\Omega}$ and in particular $|\nabla h|=1$. Moreover, on $\Omega\setminus\Omega'$ we have $h>d_0$ and $|\nabla h|\leq1$. These properties, together with the estimate~\eqref{e:bg_bound1}, give  the uniform bound
\beq\label{e:technical_L1_bound}
\int_0^1
\frac{2\epsilon}{h^3(\Gamma_\epsilon)}\;dt
\leq
\int_0^1
\frac{2\epsilon}{h^3(\Gamma_\epsilon)}|\nabla h(\Gamma_\epsilon)|^2\;dt
+
\frac{2\epsilon}{d_0^3}
\leq
C
+
\frac{2\epsilon}{d_0^3}\;.
\eeq
This proves that $\epsilon\nabla U(\Gamma_\epsilon)$ is uniformly bounded in $L^1$ because
\beq
\epsilon\nabla U(\Gamma_\epsilon)=-\frac{2\epsilon}{h^3(\Gamma_\epsilon)}\nabla h(\Gamma_\epsilon)
\eeq
and $|\nabla h|\leq1$. Since $\nabla V(\Gamma_\epsilon)$ is also uniformly bounded in $L^1$ (actually in $L^\infty$), the Euler-Lagrange equation~\eqref{e:EulerLagrange_epsilon} together with $T_1\leq\tau_\epsilon\leq T_2$ forces $\Gamma_\epsilon''$ to be uniformly bounded in $L^1$ as well. Thus, $\Gamma_\epsilon$ is uniformly bounded in $W^{2,1}$. By the compactness of the embedding $W^{2,1}(S^1;\R^N)\hookrightarrow H^1(S^1;\R^N)$, up to passing to a subsequence for $\epsilon\to0$, we have that $\Gamma_\epsilon$ converges to some $\Gamma:S^1\to\overline\Omega$ in $H^1$.

Now, since the functions $\tilde\mu_\epsilon:=2\epsilon h^{-3}(\Gamma_\epsilon)$ are uniformly bounded in $L^1$, up to passing to a subsequence for $\epsilon\to0$, $\tilde\mu_\epsilon $ converges to some $\tilde\mu $ in $L^1$ weak-$\ast$. By the Riesz representation Theorem, $\tilde\mu $ is a positive, finite Borel measure. We set
\beq\label{e:CCC_Benci}
\CCC':=\{t\in S^1\,|\,\Gamma(t)\in\partial\Omega\}\;.
\eeq 
Since, for each $t\not\in \CCC'$, the function $\tilde\mu_\epsilon $ converges uniformly to $0$ in a neighborhood of $t$ the support of $\tilde\mu $ is contained in $\CCC'$. Moreover, if $t\in\CCC'$, for $\epsilon\to0$ the sequence $\nabla h(\Gamma_\epsilon(t))$ converges to $-\nu(\Gamma(t))$. Thus, taking the limit $\epsilon\to0$ in~\eqref{e:weakEulerLagrange_epsilon} we obtain
\beq
\tau^{-2}
\int_0^1
\langle
\Gamma',\Psi'
\rangle
\,dt
-
\int_0^1
\langle
\nabla V(\Gamma),\Psi
\rangle
\,dt
=
\int_{\CCC'}
\langle
\nu(\Gamma)
,\Psi
\rangle
\, d\tilde\mu 
,
\qquad
\forall \Psi\in H^1(S^1;\R^N).
\eeq
By the reparametrization $\R/\tau\Z\to S^1$ given by $t\mapsto \frac t\tau$ the measure $\tilde\mu $ is pulled-back to a measure $\mu$ on $\CCC:=\{t\in\R/\tau\Z\,|\,\gamma(t)\in\partial\Omega\}$ and the above equation can be rewritten as in point~(i) of the statement. 

Now, if $t\not\in\supp(\mu)$, we can take $\epsilon>0$ sufficiently small such that $[t-\epsilon,t+\epsilon]\cap\supp(\mu)=\emptyset$. For each $\psi\in H^1(\R/\tau\Z;\R^N)$ supported in $[t-\epsilon,t+\epsilon]$, point~(i) reduces to
\beq
\int_{t-\epsilon}^{t+\epsilon}
\Bigl[ \langle \gamma',\psi' \rangle - \langle \nabla V(\gamma),\psi \rangle \Bigr]
dt
=0,
\eeq
and a usual bootstrap argument readily implies that $\gamma$ is a smooth solution of the Euler-Lagrange equation of $L$ on $[t-\epsilon,t+\epsilon]$. This establishes point~(ii).

Now, point~(i) also implies that  $\gamma'$ is a curve of bounded variation. Therefore $\gamma$ has left and right derivatives at each point and they are left and right continuous respectively. In order to conclude the proof, we only need to establish that the reflection rule is satisfied at each time $t\in\supp(\mu)$.

Up to passing to a subsequence for $\epsilon\to 0$, the sequence $\epsilon U(\gamma_\epsilon)$ converges to $0$ almost everywhere. Indeed, assume that $\epsilon U(\gamma_\epsilon)$  does not converge to zero on a set $I\subset\R/\tau\Z$. Then, $h(\gamma_\epsilon)\to0$ and $|\nabla h(\gamma_\epsilon)|\to1$ pointwise on $I$. Since
\beq
\epsilon \nabla U(\gamma_\epsilon)
=
\epsilon U(\gamma_\epsilon)
\frac{-2 \nabla h(\gamma_\epsilon(t))}{h(\gamma_\epsilon(t))}\;,
\eeq
then $|\epsilon \nabla U(\gamma_\epsilon)|\to+\infty$  pointwise on $I$. Now, assume that $I$ has positive Lebesgue measure. By Fatou's Lemma we get
\beq
\liminf_{\epsilon\to0}
\int_I
|\epsilon \nabla U(\gamma_\epsilon)|\,dt
\geq
\int_I
\liminf_{\epsilon\to0}
|\epsilon \nabla U(\gamma_\epsilon)|\,dt
=
+\infty,
\eeq
which contradicts the fact that $\epsilon \nabla U(\gamma_\epsilon)$ is uniformly bounded in $L^1$.

Since $\epsilon U(\gamma_\epsilon)$ converges to $0$ almost everywhere and $E_\epsilon\to E$, we have that $\frac12|\gamma'|+V(\gamma)=E$ almost everywhere, and since $\gamma'$ has bounded variation we actually obtain
\beq
\label{e:conservation_energy_lxrx}
\tfrac12|\gamma'(t^\pm)|+V(\gamma(t))=E\qquad\forall t\in\R/\tau\Z.
\eeq

Now, let us consider a time $t$ which is an isolated point in $\supp(\mu)$. In point (i) of the statement,  let us choose $\psi$ to be supported in the interval $[t-\epsilon,t+\epsilon]$, where $\epsilon>0$ is sufficiently small so that $[t-\epsilon,t+\epsilon]\cap\supp(\mu)=\{t\}$. After an integration by parts we obtain
\beq
\langle 
\gamma'(t^-)-\gamma'(t^+),\psi(t)
\rangle 
-
\int_{[t-\epsilon,t+\epsilon]\setminus\{t\}}   \langle \gamma'' + \nabla V(\gamma),\psi\rangle\, dt =  \langle \nu(\gamma(t)),\psi(t)\rangle\, \mu(\{t\})\,.
\eeq
Since $\gamma$ is a solution of the Euler-Lagrange equation of $L$ on $[t-\epsilon,t+\epsilon]\setminus\{t\}$, the integral on the left-hand side is zero and we actually have
\beq\label{e:bouncing_measure}
\langle 
\gamma'(t^-)-\gamma'(t^+),v
\rangle 
= \langle \nu(\gamma(t)),v\rangle\, \mu(\{t\}),\qquad\forall v\in\R^N\,.
\eeq
Choosing $v$ to be an arbitrary vector tangent to $\partial\Omega$ at $\gamma(t)$, namely $\langle\nu(\gamma(t)),v\rangle=0$, we obtain that the components of $\gamma'(t^-)$ and $\gamma'(t^+)$ tangent to $\partial\Omega$ are the same, i.e.
\beq
\gamma'(t^+)-\langle \nu(\gamma(t)),\gamma'(t^+) \rangle\cdot\nu(\gamma(t))
=
\gamma'(t^-)-\langle \nu(\gamma(t)),\gamma'(t^-) \rangle\cdot\nu(\gamma(t))\;.
\eeq
This, together with conservation of energy \eqref{e:conservation_energy_lxrx}, implies that 
\beq
|\langle \nu(\gamma(t)),\gamma'(t^+) \rangle|=|\langle \nu(\gamma(t)),\gamma'(t^-) \rangle|,
\eeq 
and if this latter quantity is nonzero then we must have
\beq
\langle \nu(\gamma(t)),\gamma'(t^+) \rangle=-\langle \nu(\gamma(t)),\gamma'(t^-) \rangle.
\eeq
Finally, by choosing $v=\nu(\gamma(t))$ in  equation~\eqref{e:bouncing_measure}  we obtain
\beq
 \langle \nu(\gamma(t)),\gamma'(t^+) \rangle
=
\tfrac12  \langle
\gamma'(t^-)-\gamma'(t^+),\nu(\gamma(t))
\rangle 
= 
\tfrac12
\mu(\{t\})\neq 0.
\eeq
This concludes the proof of point~(iii).
\end{proof}

\begin{Prop}\label{prop:bounde_Morse_index}
We consider the situation of Proposition~\ref{prop:Benci_improved}. Then, up to taking a subsequence of $\{(\Gamma_\epsilon,\tau_\epsilon)\}$,  the cardinality $|\supp(\mu)|$ of the support of the measure $\mu$ is bounded from above by the Morse index of the restricted functional $\L^{E_\epsilon}_\epsilon|_{H^1\times\{\tau_\epsilon\}}$ at $\Gamma_\epsilon$ for all $\epsilon$ sufficiently small, i.e.
\beq
|\supp(\mu)|
\leq
\liminf_{\epsilon\to0}
\Morse\big(\Gamma_\epsilon; \L^{E_\epsilon}_\epsilon|_{H^1\times\{\tau_\epsilon\}}\big)\;.
\eeq
\end{Prop}

\begin{proof}
With the notation adopted in the proof of Proposition~\ref{prop:Benci_improved} (see in particular the paragraph of equation~\eqref{e:CCC_Benci}), the measure $\mu$ is the pullback of a measure $\tilde\mu$ on $S^1=\R/\Z$ via the reparametrization $\iota:\R/\tau\Z\to S^1$ given by $\iota(t)=\tfrac t\tau$. In particular $\iota(\supp(\mu))=\supp(\tilde\mu)$ and
\beq
|\supp(\mu)|=|\supp(\tilde\mu)|.
\eeq
Hence, in order to prove the proposition it is enough to establish the following: for each point $t\in\supp(\tilde\mu)$ and for each $\epsilon>0$ sufficiently small, there exists a vector field $\Psi_\epsilon\in H^1(S^1;\R^N)$ supported on an sufficiently small neighborhood of $t$ and such that 
\beq\label{e:negative_hessian}
\hess \L^{E_\epsilon}_\epsilon(\Gamma_\epsilon,\tau_\epsilon) [ (\Psi_\epsilon,0),(\Psi_\epsilon,0) ]
<0.
\eeq
In fact, assume that this is verified. Then, for $k$ distinct points $t_1,...,t_k\in\supp(\tilde\mu)$ and sufficiently small $\epsilon>0$ we can find $k$ vector fields $\Psi_{\epsilon,1},...,\Psi_{\epsilon,k}$ such that each $\Psi_{\epsilon,j}$ is supported in a sufficiently small neighborhood of $t_j$ and verifies~\eqref{e:negative_hessian}. In particular, we may assume that the supports of the $\Psi_{\epsilon,j}$'s are pairwise disjoint. Therefore, these vector fields span a $k$-dimensional vector subspace of $H^1(S^1;\R^N)$ over which the Hessian of the restricted action functional\footnote{Notice that $\hess \L^{E_\epsilon}_\epsilon|_{H^1\times\{\tau_\epsilon\}}(\Gamma_\epsilon) [ \Psi,\Xi]=\hess \L^{E_\epsilon}_\epsilon(\Gamma_\epsilon,\tau_\epsilon) [ (\Psi,0),(\Xi,0) ]$.} $\L^{E_\epsilon}_\epsilon|_{H^1\times\{\tau_\epsilon\}}$ at $\Gamma_\epsilon$ is negative definite, which implies 
\beq
\Morse\big(\Gamma_\epsilon; \L^{E_\epsilon}_\epsilon|_{H^1\times\{\tau_\epsilon\}}\big)\geq k.
\eeq

Let us now establish the assertion made at the beginning of the proof. From now on we fix $t\in\supp(\tilde\mu)$ and $\epsilon>0$ sufficiently small. For $\delta>\delta'>0$ we choose a smooth function $\phi_\epsilon:\R/\Z\to[0,1]$ such that $\supp(\phi_\epsilon)\subseteq[t-\delta,t+\delta]$ and $\phi_\epsilon\equiv1$ on $[t-\delta',t+\delta']$. We define the vector field $\Psi_\epsilon\in H^1(S^1;\R^N)$ by
\beq
\Psi_\epsilon(s):=-\phi_\epsilon(s) \nabla h(\Gamma_\epsilon(s)).
\eeq
We will show that $\Psi_\epsilon$ satisfies~\eqref{e:negative_hessian}. The left-hand side of~\eqref{e:negative_hessian} computes to
\bea
&\hess \L^{E_\epsilon}_\epsilon(\Gamma_\epsilon,\tau_\epsilon)[(\Psi_\epsilon,0),(\Psi_\epsilon,0)]\\
&\qquad
=
\tau_\epsilon
\int_0^1
\Big[
\tau_\epsilon^{-2}
\langle
\partial_{vv}L_\epsilon(\Gamma_\epsilon,\tfrac1{\tau_\epsilon}\Gamma_\epsilon') \Psi'_\epsilon,\Psi'_\epsilon
\rangle
+
\langle
\partial_{qq}L_\epsilon(\Gamma_\epsilon,\tfrac1{\tau_\epsilon}\Gamma_\epsilon') \Psi_\epsilon,\Psi_\epsilon
\rangle
\Big]
ds
\\
&\qquad
= A_\epsilon - B_\epsilon,
\eea
where
\bea
A_\epsilon &=
\tau_\epsilon
\int_0^1
\left[
\tau_\epsilon^{-2}
|\Psi'_\epsilon|^2
-
\langle
\nabla^2V(\Gamma_\epsilon)
\Psi_\epsilon,\Psi_\epsilon
\rangle
+
2\epsilon\frac{\langle\nabla^2h(\Gamma_\epsilon)\Psi_\epsilon,\Psi_\epsilon\rangle}{h^3(\Gamma_\epsilon)}
\right]dt,
\\
B_\epsilon &=
6\tau_\epsilon
\epsilon
\int_0^1
\frac{\langle\nabla h(\Gamma_\epsilon),\Psi_\epsilon\rangle^2}{h^4(\Gamma_\epsilon)}\;
ds.
\eea
Now, the term $|A_\epsilon|$ is uniformly bounded  in $\epsilon$. Indeed, since $\Gamma_\epsilon$ converges in $H^1$, the vector field $\Psi_\epsilon$ is uniformly bounded in $H^1$, which implies that $\Psi_\epsilon'$ is uniformly bounded in $L^2$. Moreover, in the proof of Proposition~\ref{prop:Benci_improved} (see equation~\eqref{e:technical_L1_bound}) we showed that $2\epsilon h^{-3}(\Gamma_\epsilon)$ is uniformly bounded in $L^1$, and therefore the last summand under the integral in $A_\epsilon$ is also uniformly bounded in $L^1$.

As for $B_\epsilon$, we want to show that it goes to $+\infty$ as $\epsilon\to0$. Since $\Gamma_\epsilon\to\Gamma$ in $H^1$ (in particular in $C^0$) as $\epsilon\to0$ and $|\nabla h|=1$ on $\partial\Omega$, we can find $\delta''\in(0,\delta']$ such that $|\nabla h(\Gamma_\epsilon(s))|^4\geq \tfrac12$ for each $s\in[t-\delta'',t+\delta'']$ and $\epsilon>0$ sufficiently small. Therefore we can estimate
\bea
B_\epsilon
&
\geq
6 T_1 \epsilon
\int_{t-\delta''}^{t+\delta''}
\frac{\langle\nabla h(\Gamma_\epsilon),\Psi_\epsilon\rangle^2}{h^4(\Gamma_\epsilon)}\;
ds\\
&
=
6 T_1 \epsilon
\int_{t-\delta''}^{t+\delta''}
\frac{|\nabla h(\Gamma_\epsilon)|^4}{h^4(\Gamma_\epsilon)}\;
ds\\
&
\geq
\frac{6 T_1 \epsilon}{2}
\int_{t-\delta''}^{t+\delta''}
\frac{1}{h^4(\Gamma_\epsilon)}\;
ds\\
(\text{by H\"older ineq.})&
\geq
\frac{6 T_1 \epsilon}{2(2\delta'')^{1/3}}
\biggl(
\int_{t-\delta''}^{t+\delta''}
\frac{1}{h^3(\Gamma_\epsilon)}\;
ds
\biggr)^{4/3}\\
&
=\frac{6 T_1}{2(2\delta'')^{1/3}}
\biggl(\,\underbrace{
\int_{t-\delta''}^{t+\delta''}
\frac{\epsilon}{h^3(\Gamma_\epsilon)}\;
ds}_{\displaystyle\qquad =:B_\epsilon'}\,
\biggr)
\biggl(\,
\underbrace{\int_{t-\delta''}^{t+\delta''}
\frac{1}{h^3(\Gamma_\epsilon)}\;
ds}_{\displaystyle\qquad =:B_\epsilon''}
\,\biggr)^{1/3}
\;.
\eea
As we showed in the proof of Proposition~\ref{prop:Benci_improved} (see the paragraph of equation~\eqref{e:CCC_Benci}), up to a subsequence the function $2\epsilon h^{-3}(\Gamma_\epsilon)$ converges to the measure $\tilde\mu$ in $L^1$ weak-$\ast$, which implies that $B_\epsilon'$ converges to a constant $B'\geq \tfrac12\tilde\mu(\{t\})>0$. Hence, it remains to be shown that $B_\epsilon''\to+\infty$ as $\epsilon\to0$. By the H\"older inequality we get
\beq
B_\epsilon''
=
\int_{t-\delta''}^{t+\delta''}
\frac{1}{h^3(\Gamma_\epsilon)}\;
ds
\geq 
(2\delta'')^{-1/2}
\biggl(\,
\underbrace{\int_{t-\delta''}^{t+\delta''}
\frac 1{h^2(\Gamma_\epsilon)}\,
ds}_{\displaystyle\qquad =:B_\epsilon'''}
\,
\biggr)^{3/2}.
\eeq
We recall that up to a subsequence $\Gamma_\epsilon\to\Gamma$ as $\epsilon\to0$ in $H^1$, and that $\Gamma(t)\in\partial\Omega$. By the definition of $h$, if we choose $\delta''$ small enough we have that
\beq
h(\Gamma_\epsilon(s))=\mathrm{dist}_{\partial\Omega}(\Gamma_\epsilon(s)),\qquad \forall s\in[t-\delta'',t+\delta''].
\eeq
for all $\epsilon>0$ sufficiently small. Then, let $D>0$ be a uniform upper bound for the $L^2$ norm of the vector fields $\Gamma_\epsilon'$. For each $ s\in[t-\delta'',t+\delta'']$ we can estimate using $|\nabla h|\leq1$
\beq
|h(\Gamma_\epsilon(s))-h(\Gamma_\epsilon(t))|
\leq
|\Gamma_\epsilon(s)-\Gamma_\epsilon(t)|
\leq
|s-t|^{1/2} \|\Gamma_\epsilon'\|_{L^2}
\leq
|s-t|^{1/2} D.
\eeq
This implies
\bea
B_\epsilon'''
&
=
\int_{t-\delta''}^{t+\delta''}
\frac 1{h^2(\Gamma_\epsilon(s))}\,
ds
\geq
\int_{t-\delta''}^{t+\delta''}
\frac 1{(h(\Gamma_\epsilon(0))+|s-t|^{1/2} D)^2 }\,
ds\\
&=
\int_{-\delta''}^{\delta''}
\frac 1{(h(\Gamma_\epsilon(0))+|s|^{1/2} D)^2 }\,
ds
=
2\int_{0}^{\delta''}
\frac 1{(h(\Gamma_\epsilon(0))+s^{1/2} D)^2 }\,
ds\\
&\geq
\int_{0}^{\delta''}
\frac 1{h^2(\Gamma_\epsilon(0))+s D^2 }\,
ds
=\frac1{D^2} \ln\left(1+ \frac{D^2\delta''}{h(\Gamma_\epsilon(0))}\right).
\eea
Since up to a subsequence for $\epsilon\to0$ we have $h(\Gamma_\epsilon(0))\to0$, from the above estimate we infer that $B_\epsilon'''\to+\infty$. Thus, this shows that $B_\epsilon\to+\infty$ and therefore the proposition follows.
\end{proof}

Next, we examine the case in Proposition \ref{prop:Benci_improved} where the periods go to zero.

\begin{Prop}\label{prop:Benci_improved_without_lower_period_bound}
Let $K>0$ and $(\Gamma_\epsilon,\tau_\epsilon)$ be a critical point of $\L^{E_\epsilon}_\epsilon$ with  $E_\epsilon\leq K$ and $\tau_\epsilon\to0$ as $\epsilon\to0$. Then, up to a subsequence for $\epsilon\to0$, $\Gamma_\epsilon$ converges in $C^0$ to a constant curve $\gamma\equiv q\in\overline{\Omega}$. Moreover, one of the following holds.
\begin{itemize}
\item[(i)] $q$ is a critical point of the potential $V$.
\item[(ii)] $q$ lies in $\partial\Omega$ and there exists $a>0$ such that $\nabla V(q)=-a\nu(q)$, where $\nu$ is the outer normal to $\partial\Omega$.
\end{itemize}
\end{Prop}

\begin{Rmk}
In case (ii) of Proposition \ref{prop:Benci_improved_without_lower_period_bound} the stationary curve $\gamma(t)\equiv q$ describes a particle confined by the potential, see Figure \ref{fig:case2}.
\begin{figure}[hbt]
\frag{V}{$V$}
\frag{G}{$-\nabla V(q)$}
\frag{w}{potential wall}
\frag{p}{particle at $q\in\partial\Omega$}
\includegraphics[scale=1]{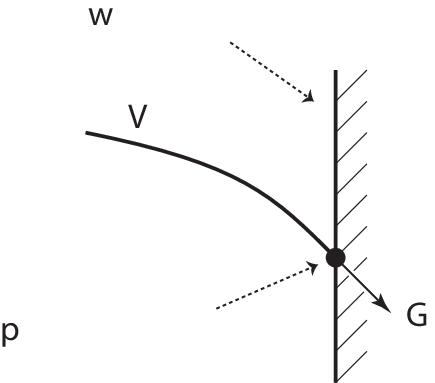}
\caption{}\label{fig:case2}
\end{figure}
\end{Rmk}

\begin{proof}
We choose a sequence of positive integers $\{\kappa_\epsilon\}$ such that $T_1<\kappa_\epsilon\tau_\epsilon<T_2$ for suitable $T_2>T_1>0$ and we define $(\Theta_\epsilon,\sigma_\epsilon)\in H^1(S^1;\overline\Omega)\times\R_{>0}$ by $\Theta_\epsilon(t):=\Gamma_\epsilon(\kappa_\epsilon t)$ and $\sigma_\epsilon:=\kappa_\epsilon\tau_\epsilon$. We point out that $(\Theta_\epsilon,\sigma_\epsilon)$ is a  critical point of the action functional $\L^{E_\epsilon}_\epsilon$. By Proposition~\ref{prop:Benci_improved} we conclude that, up to a subsequence, $(\Theta_\epsilon,\sigma_\epsilon)\to(\Theta,\sigma)$ in $H^1(S^1,\overline{\Omega})\times\R_{>0}$ as $\epsilon\to0$. In particular $\Theta_\epsilon\to\Theta$ in $C^0$.

We claim that $\Theta$ is a constant curve. Indeed, let us assume by contradiction that there exist $t_1<t_2$ such that 
\beq\label{e:Theta_not_constant_by_contradiction}
|\Theta(t_1)-\Theta(t_2)|>0
\eeq
Notice that each $\Theta_\epsilon$ is ${\kappa_\epsilon}^{-1}$ periodic, and in particular
\beq
\Theta_\epsilon(t_2)=\Theta_\epsilon(t_2-j \kappa_\epsilon^{-1}),\qquad \forall j\in\N\,. 
\eeq
Since $\kappa_\epsilon\to\infty$, we can find a sequence of positive integers $\{j_\epsilon\}$ such that $j_\epsilon \kappa_\epsilon^{-1}\to t_2 - t_1$. This, together with the $C^0$ convergence $\Theta_\epsilon\to\Theta$, implies
\beq
\Theta(t_1)
=
\lim_{\epsilon\to0}
\Theta_\epsilon(t_1)
=
\lim_{\epsilon\to0}
\Theta_\epsilon(t_2-j_\epsilon \kappa_\epsilon^{-1})
=
\lim_{\epsilon\to0}
\Theta_\epsilon(t_2)
=
\Theta(t_2),
\eeq
which contradicts~\eqref{e:Theta_not_constant_by_contradiction}.

Since each curve $\Theta_\epsilon$ is an iteration of $\Gamma_\epsilon$, the fact that $\Theta_\epsilon$  converges in $C^0$ to a constant curve forces $\Gamma_\epsilon$  to converge in $C^0$ to the same constant curve $\Gamma=\Theta\equiv q\in\overline{\Omega}$. Then,  the integral equation in point~(i) of Proposition~\ref{prop:Benci_improved}  reduces to
\beq
-\int_0^\sigma \langle \nabla V(q),\psi \rangle\, dt = \int_{\CCC} \langle \nu(q),\psi \rangle\, d\mu
\qquad
\forall \psi\in C^\infty(\R/\sigma\Z;\R^N).
\eeq
Here,  $\CCC=\emptyset$ if $q\in\Omega$ and $\CCC=\R/\sigma\Z$ if $q\in\partial\Omega$. This immediately implies the proposition.
\end{proof}

\section{Proof of Theorem~\ref{thm:main}}\label{sec:proof_main}

In order to prove Theorem~\ref{thm:main}, it only remains to build the sequence $\{(\Gamma_\epsilon,\tau_\epsilon)\}$ of critical points of the free-time action functionals $\{\L^{E_\epsilon}_\epsilon\}$, as needed in Proposition~\ref{prop:Benci_improved}. This will be carried out in the Hamiltonian formulation, by considering  the Hamiltonian function
\bea
H_\epsilon:T^*\Omega=\Omega\times\R^N&\to\R\\
(q,p)&\mapsto\tfrac12|p|^2+V(q)+\epsilon U(q)\;.
\eea
This Hamiltonian is the Legendre-dual to the Lagrangian $L_\epsilon$ defined in~\eqref{e:Lagrangian_epsilon}. By Legendre duality, $\tau$-periodic Hamiltonian orbits $v:\R/\tau\Z\to\Omega\times\R^N$ of $H_\epsilon$ with energy $H_\epsilon(v)=E$ are in one-to-one correspondence to $\tau$-periodic  solutions $\gamma=\pi(v)$ of the Euler-Lagrange system of $L_\epsilon$ with energy $E_\epsilon(\gamma)=E$ via the projection $\pi:T^*\Omega\to\Omega$.

We begin with the following
\begin{Lemma}\label{lemma:E_above_maxV_regular}
Any energy value $E>\max_{\overline{\Omega}} V$ is a regular value of the Hamiltonian function $H_\epsilon$ provided $\epsilon>0$ is sufficiently small.  
\end{Lemma}

\begin{proof}
Since $H_\epsilon$ is a classical Hamiltonian (i.e.\ of the form kinetic energy plus potential), the energy hypersurface $\Sigma_\epsilon$ is regular provided the boundary of its projection into the base, i.e.\ the set 
\beq
\Upsilon_\epsilon=\partial\,\pi(\Sigma_\epsilon)
=\{V+\epsilon U=E\}\subset\R^N\,,
\eeq  
does not contain any critical point of the potential $V+\epsilon U$. This is always verified if $\epsilon$ is sufficiently small. Indeed, for $q\in\Upsilon_\epsilon$ we have by~\eqref{e:definition_of_U}
\beq\label{e:h_on_Upsilon}
h^2(q)=\frac{\epsilon}{E-V(q)}\,,
\eeq
and therefore
\bea\label{e:estimating_V+eU}
|\nabla V(q)+\epsilon\nabla U(q)|
&\geq
|\epsilon\nabla U(q)|-|\nabla V(q)|\\
&=
\frac{2\epsilon}{h^3(q)}\, |\nabla h(q)|
-|\nabla V(q)|\\
&=
2\epsilon^{-1/2}(E-V(q))^{3/2}|\nabla h(q)|
-|\nabla V(q)|\\
&\geq
2\epsilon^{-1/2}(E-{\textstyle\max_{\overline\Omega}}V)^{3/2}|\nabla h(q)|
-|\nabla V(q)|\,.
\eea
Equation~\eqref{e:h_on_Upsilon} implies that $h|_{\Upsilon_\epsilon}\to0$ uniformly as $\epsilon\to0$. Hence, for sufficiently small $\epsilon$ we have $h|_{\Upsilon_\epsilon}=\mathrm{dist}_{\partial\Omega}|_{\Upsilon_\epsilon}$ and  $|\nabla h|\geq \tfrac12$ on $\Upsilon_\epsilon$. Combining this with~\eqref{e:estimating_V+eU} we obtain
\beq
|\nabla V(q)+\epsilon\nabla U(q)|
\geq
\epsilon^{-1/2}
(E-{\textstyle\max_{\overline\Omega}}V)^{3/2}
-
|\nabla V(q)|,\qquad\forall q\in\Upsilon_\epsilon,
\eeq
from which we conclude that $\nabla V+\epsilon\nabla U$ does not vanish on $\Upsilon_\epsilon$ for $\epsilon$ sufficiently small.
\end{proof}

From now on we fix an energy value $E>\max_{\overline{\Omega}}V$ and we consider $\epsilon>0$ small enough so that Lemma~\ref{lemma:E_above_maxV_regular} holds. In particular, the energy hypersurface
\beq
\Sigma_\epsilon:=\{H_\epsilon=E\}
\eeq 
is a smooth and non-empty closed manifold. Notice that $\pi:T^*\Omega\to\Omega$ projects $\Sigma_\epsilon$ into the compact set $\Omega_\epsilon:=\{\epsilon U\leq E\}$. Thus, we can modify the potential $\epsilon U$ outside $\Omega _\epsilon $ and extend it to a global potential  $U_\epsilon\in C^\infty(\R^N)$ such that $U_\epsilon=\epsilon U$ on $\Omega_{\epsilon/2}$, $U_\epsilon>E$ outside $\Omega_{\epsilon/2}$ and $U\equiv E'> E$ outside $\Omega$, see Figure \ref{fig:graph}. Analogously, we extend $V$ to a compactly supported function $V\in C^\infty(\R^N)$ such that $V>-(E'-E)$. In particular $V+U_\epsilon>E$ outside $\Omega_\epsilon$, and $V+U_\epsilon\equiv E'$ outside a compact neighborhood of $\overline\Omega$.

\frag[n]{H}{$\epsilon U$}
\frag[n]{h}{$U_\epsilon$}
\frag[n]{Hh}{$\epsilon U\equiv U_\epsilon$}
\frag[n]{e}{$E$}
\frag[n]{f}{$E'$}
\frag[n]{o}{$\partial\Omega_\epsilon$}
\frag[n]{u}{$\partial\Omega_{\epsilon/2}$}
\frag[n]{O}{$\Omega$}
\frag[n]{0}{$0$}
\begin{figure}[h]
\includegraphics[scale=1]{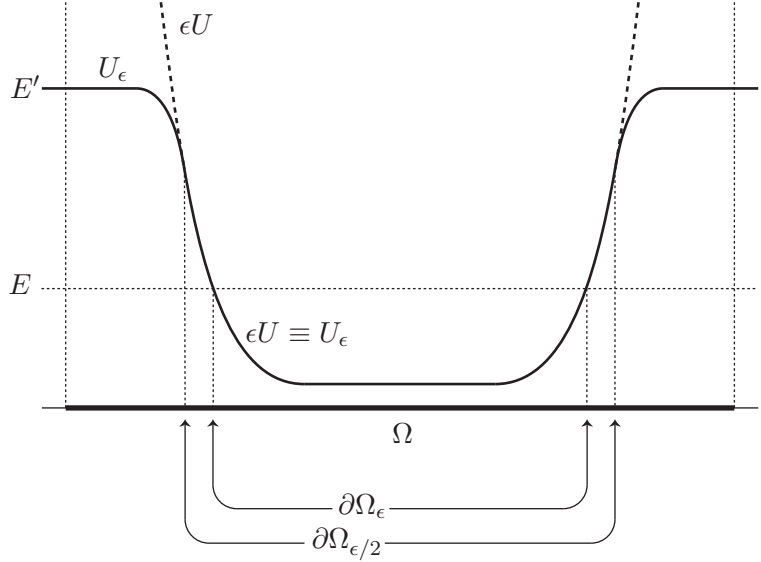}
\caption{The potential $\epsilon U$ and its modification $U_\epsilon$.}\label{fig:graph}
\end{figure}

For technical reasons we compactify $\R^N$ to $S^N$ in such a way that 
\beq
\mathrm{vol}(S^N)\gg\mathrm{vol}(\Omega), 
\eeq
and we further extend $U_\epsilon$ and $V$ to  smooth functions on $S^N$ that we still denote by $U_\epsilon$ and $V$. Finally, we introduce the modified Hamiltonian
\bea
K_\epsilon:T^*S^N&\to\R\\
(q,p)&\mapsto\tfrac12|p|^2+V(q)+ U_\epsilon(q)\;.
\eea
Notice that $\Sigma_\epsilon =\{K_\epsilon=E\}$ and the Hamiltonian flows of $H_\epsilon$ and $K_\epsilon$ agree on $\Sigma_\epsilon$. 

Since $K_\epsilon$ is a classical Hamiltonian a well-known result in Hamiltonian dynamics asserts that the energy hypersurface $\Sigma_\epsilon$ is of restricted contact type, i.e.~there exists a primitive $\lambda_\epsilon$ of the canonical symplectic form $\omega$ of $T^*S^N$ such that $\lambda_\epsilon|_{\Sigma_\epsilon}$ is a contact form. We recall that a primitive $\lambda_\epsilon$ of $\omega$ restricts to a contact form on $\Sigma_\epsilon$ if and only if the associated Liouville vector field $P_\epsilon$, defined by $\omega(P_\epsilon,\cdot)=\lambda_\epsilon$, is transverse to $\Sigma_\epsilon$. This is equivalent to asking that $\lambda_\epsilon(X_\epsilon)\neq0$, since 
\beq
\lambda_\epsilon(X_\epsilon)
=
\omega(P_\epsilon,X_\epsilon)
=
dK_\epsilon(P_\epsilon).
\eeq
For later purposes, we need to show that we can choose  $\lambda_\epsilon$ such that $\lambda_\epsilon(X_\epsilon)$ is bounded away from zero uniformly in $\epsilon$.

\begin{Prop}\label{prop:uniform_contact_type_estimate}
We fix $E>\max_{\overline\Omega} V$. For $\epsilon>0$ small enough there exists a 1-form $\lambda_\epsilon$ on $T^*S^N$ with $d\lambda_\epsilon=\om$ which restricts to a contact form on $\Sigma_\epsilon=\{K_\epsilon=E\}$. Moreover, on $\Sigma_\epsilon$ we have the estimate
\beq
\lambda_\epsilon(X_\epsilon)
\geq
\frac{(E-\textstyle\max_{\overline\Omega}V)^3}{2\bigl[(E-\textstyle\max_{\overline\Omega}V)^2+ 48(E-{\textstyle\min_{\overline\Omega}}V)^2\bigr]}
=:\Lambda(E)>0\;.
\eeq
\end{Prop}

\begin{proof}
We denote by $\lambda=\sum_ip_idq_i$ the Liouville 1-form on $T^*S^N$. The Hamiltonian vector field $X_\epsilon$ of $K_\epsilon$ is given in local coordinates
by
\beq
X_\epsilon=\sum_{i} 
\left[
p_i\, \frac{\partial}{\partial q_i}
-
\biggl(
\frac{\partial U_\epsilon}{\partial q_i}
+
\frac{\partial V}{\partial q_i}
\biggr)\,
\frac{\partial}{\partial p_i}
\right].
\eeq
Thus, we have $\lambda(X_\epsilon)=|p|^2\geq0$. Now we consider $u:T^*\Omega\to\R$ given by 
\beq
u(q,p)= dU(q)[p]=\sum_i\frac{\p U}{\p q_i}(q)p_i 
\eeq
and define the 1-form 
\beq
\lambda_\epsilon
:=
\lambda-C\epsilon\,du
=
\lambda 
- 
C\epsilon \sum_{i,j} \frac{\partial^2 U_\epsilon}{\partial q_i\, \partial q_j}\,p_i\,dq_j
-
C\epsilon \sum_{i} \frac{\partial U_\epsilon}{\partial q_i}\,dp_i,
\eeq 
where $C>0$ is a constant independent of $\epsilon$ that we will fix later. Since $U_\epsilon=\epsilon U$ on $\Sigma_\epsilon$ and using the definition of $U$ (see~\eqref{e:definition_of_U}), the function $\lambda_\epsilon(\Sigma_\epsilon)$ on $\Sigma_\epsilon$ is given by
\bea
\lambda_\epsilon(X_\epsilon) \bigr|_{\Sigma_\epsilon}
=\, &
|p|^2 
- 
C\epsilon^2\, \hess U(q)[p,p]+
C\epsilon^3\, |\nabla U|^2+
C\epsilon^2\, \langle\nabla U,\nabla V\rangle 
\\
=\, &
|p|^2
+
2 C\epsilon^2
h^{-3}\,\hess h(q)[p,p]
+
4C\epsilon^3 h^{-6}\, |\nabla h|^2
\\
& -
6C\epsilon^2 
h^{-4}\,|dh(q)p|^2
-2C\epsilon^2 h^{-3} \langle\nabla h,\nabla V\rangle.
\eea
Now we notice that for $(q,p)\in\Sigma_\epsilon$ we have
\beq\label{e:h(q)_funct_of_p}
h^2(q)=\frac{\epsilon}{E-V(q)-\tfrac12|p|^2}\;.
\eeq
We choose $\kappa\geq0$ such that
\beq
\hess h(q)[p,p]\geq-\kappa|p|^2\quad\forall q\in{\Omega}.
\eeq
Then using $|\nabla h|\leq 1$ we have the estimate
\bea
\lambda_\epsilon(X_\epsilon)\bigr|_{\Sigma_\epsilon}
\geq\,&
|p|^2 \bigl( 1 - 6C\epsilon^2 h^{-4} \bigr)
-2 C\epsilon^2h^{-3}\kappa|p|^2
+4C\epsilon^3 h^{-6}\, |\nabla h|^2
-2C\epsilon^2 h^{-3} \langle\nabla h,\nabla V\rangle\\
=\,&
|p|^2 \left( 1 - 6C(E-V(q)-\tfrac12|p|^2)^2 \right)
-2 C\epsilon^{1/2}\kappa|p|^2(E-V(q)-\tfrac12|p|^2)^{3/2}\\
&+
4C (E-V(q)-\tfrac12|p|^2)^3\, |\nabla h|^2
-2C\epsilon^{1/2} (E-V(q)-\tfrac12|p|^2)^{3/2} \langle\nabla h,\nabla V\rangle\\
=\,&
|p|^2 \left( 1 - 6C(E-V(q)-\tfrac12|p|^2)^2 \right)
+
4C (E-V(q)-\tfrac12|p|^2)^3\, |\nabla h|^2\\
&-2C\epsilon^{1/2} (E-V(q)-\tfrac12|p|^2)^{3/2}\Big[\kappa|p|^2+\langle\nabla h,\nabla V\rangle\Big]\,.
\eea
Now, we require $C\equiv C(E)>0$ to satisfy
\beq
6C(E-{\textstyle\min_{\overline\Omega}}V)^2<1,
\eeq
and estimate further
\bea\label{e:estimate_alpha_X}
\lambda_\epsilon(X_\epsilon)\bigr|_{\Sigma_\epsilon}
\geq\,&
|p|^2 \left( 1 - 6C(E-{\textstyle\min_{\overline\Omega}}V)^2 \right)
+
4C\,(E-{\textstyle\max_{\overline\Omega}} V-\tfrac12|p|^2)^3\, |\nabla h|^2
-
c_\epsilon\,
,
\eea
where 
\beq
c_\epsilon:=2C\epsilon^{1/2} (E-{\textstyle\min_{\overline\Omega}}V)^{3/2}\big[2\kappa(E-{\textstyle\min_{\overline\Omega}}V)+ {\textstyle\max_{\overline\Omega}}|\nabla V|\big] \longrightarrow 0\qquad\mbox{as }\epsilon\to0\,.
\eeq
For $\epsilon$ small enough, equation~\eqref{e:h(q)_funct_of_p} and the definition of $h$ implies that we have  $|\nabla h|\geq1/2$ in the region $\Sigma_\epsilon\cap\{|p|^2\leq E-\max_{\overline\Omega}V \}$. Thus, \eqref{e:estimate_alpha_X} implies that on $\Sigma_\epsilon\cap\{|p|^2\leq E-\textstyle\max_{\overline\Omega}V\}$
\bea
\lambda_\epsilon(X_\epsilon)
&\geq
C(E-{\textstyle\max_{\overline\Omega}} V- \tfrac12(E-\textstyle\max_{\overline\Omega}V))^3\, 
-
c_\epsilon\\
&=\tfrac18C(E-\textstyle\max_{\overline\Omega}V)^3\, 
-
c_\epsilon\;.
\eea
On $\Sigma_\epsilon\cap\{|p|^2\geq E-\textstyle\max_{\overline\Omega}V\}$ we can estimate
\bea
\lambda_\epsilon(X_\epsilon)
\geq
(E-\textstyle\max_{\overline\Omega}V)\left( 1 - 6C(E-{\textstyle\min_{\overline\Omega}}V)^2 \right)
-c_\epsilon\;.
\eea
Since $c_\epsilon\to0$ as $\epsilon\to0$,  for sufficiently small $\epsilon$ we have
\bea
\lambda_\epsilon(X_\epsilon)\bigr|_{\Sigma_\epsilon}
\geq\tfrac12\min\Big\{\tfrac18C(E-\textstyle\max_{\overline\Omega}V)^3,
(E-\textstyle\max_{\overline\Omega}V)\left( 1 - 6C(E-{\textstyle\min_{\overline\Omega}}V)^2 \right)\Big\}
\eea
Hence, by setting
\bea
C:=\frac{8}{(E-\textstyle\max_{\overline\Omega}V)^2+ 48(E-{\textstyle\min_{\overline\Omega}}V)^2}
\eea
we obtain 
\beq
\lambda_\epsilon(X_\epsilon)\bigr|_{\Sigma_\epsilon}\geq\frac{(E-\textstyle\max_{\overline\Omega}V)^3}{2\bigl[(E-\textstyle\max_{\overline\Omega}V)^2+ 48(E-{\textstyle\min_{\overline\Omega}}V)^2\bigr]}>0
\eeq
\end{proof}

Let $R_\epsilon$ be the Reeb vector field on $\Sigma_\epsilon$ associated to the contact form $\lambda_\epsilon|_{\Sigma_\epsilon}$. The above proposition implies that $X_\epsilon=r_\epsilon R_\epsilon$ where $r_\epsilon:\Sigma_\epsilon\to\R_{>0}$ is a smooth function that is bounded from below by $\Lambda(E)$. In particular, the periodic orbits of $X_{\epsilon}$ and $R_{\epsilon}$ agree up to reparamentrization. More precisely, if $v$ is a Reeb orbit of period $T$ then the corresponding orbit of $X_{\epsilon}$ has period $\tau_\epsilon$ satisfying
\beq
\tau_\epsilon\cdot \Lambda(E)\leq T\;.
\eeq 

Since $\pi(\Sigma_\epsilon)\subset\Omega$ under the projection $\pi:T^*S^N\to S^N$ the energy hypersurface $\Sigma_\epsilon$ is Hamiltonianly displaceable, that is, there exists a Hamiltonian diffeomorphism $\phi_G\in\Ham_c(T^*S^N)$ generated by a compactly supported Hamiltonian function $G:S^1\times T^*S^N\to\R$ such that
\beq
\phi_G(\Sigma_\epsilon)\cap\Sigma_\epsilon=\emptyset\;.
\eeq
In fact, let $a:S^N\to\R$ be any function which has no critical points in $\overline{\Omega}$. If we extend $a$ to $A:=a\circ\pi:T^*S^N\to\R$ then the Hamiltonian flow of $A$ displaces any compact subset of $T^*S^N|_{\overline{\Omega}}$, in particular, $\Sigma_\epsilon$. Thus, if we cut off $A$ near infinity we obtain a displacing Hamiltonian diffeomorphism in $\Ham_c(T^*S^N)$.

We recall that the displacement energy $e(\Sigma_\epsilon)$ is defined as
\beq
e(\Sigma_\epsilon):= 
\inf\left\{\int_0^1\Big[\max_{T^*S^N}G(t,\cdot)-\min_{T^*S^N}G(t,\cdot)\Big]dt\ \bigg|\ \phi_G(\Sigma_\epsilon)\cap\Sigma_\epsilon=\emptyset\right\}\;.
\eeq
\begin{Lemma}\label{lem:bound_on_displacement_energy}
The displacement energy of $\Sigma_\epsilon$ can be bounded as follows
\beq
\textstyle
e(\Sigma_\epsilon)\leq2(2E-2\min_{\overline{\Omega}} V)^{1/2}\cdot\mathrm{diam}(\Omega)
\eeq
Here $\mathrm{diam}(\Omega)$ denotes the diameter of $\Omega\subset\R^N$ and $E$ is the energy value we fixed.
\end{Lemma}

\begin{proof}
We recall that $\Sigma_\epsilon=\big\{\tfrac12|p|^2+V(q)+U_\epsilon(q)=E\big\}$. We set $R:=(2E-2\min_{\overline{\Omega}} V)^{1/2}$. Then using $U_\epsilon(q)\geq0$ we have
\beq
\textstyle\Sigma_\epsilon\subset\Omega\times B_R\subset\R^{2N}
\eeq
where $B_R\subset\R^N$ is the ball around $0$ of radius $R$. To estimate the displacement energy we choose a vector $v\in\R^N$ such that $(v+\Omega)\cap\Omega=\emptyset$ and set $G(q,p):=\sum v_ip_i:\R^{2N}\to\R$. Thus, the corresponding Hamiltonian diffeomorphism is $\phi_G(q,p)=(q+v,p)$. In particular, $\phi_G$ displaces $\Sigma_\epsilon$ from itself. To get a compactly supported Hamiltonian function we cut off $G$ to zero outside an arbitrarily small neighborhood of $\Omega\times B_R$. Thus, for any $\delta>0$ we can estimate
\bea
e(\Sigma_\epsilon)&\leq e(\Omega\times B_R)\\
&\leq\int_0^1\Big[\max_{\Omega\times B_R}G(t,\cdot)-\min_{\Omega\times B_R}G(t,\cdot)\Big]dt + \delta\\
&=\max_{\Omega\times B_R}G(q,p)-\min_{\Omega\times B_R}G(q,p) + \delta\\
&\leq \max_{\Omega\times B_R}|v|\cdot|p|-\min_{\Omega\times B_R}|v|\cdot|p| + \delta\\
&= 2R|v| + \delta\;.
\eea
Choosing the optimal vector $v$ together with the definition of $R=(2E-2\min_{\overline{\Omega}} V)^{1/2}$ proves the Lemma.
\end{proof}
The following theorem was proved by Schlenk in \cite{Schlenk_Applications_of_Hofers_geometry_to_Hamiltonian_dynamics}, see also \cite[Theorem 4.9]{Cieliebak_Frauenfelder_Paternain_Symplectic_Topology}.
\begin{Thm}
$\Sigma_\epsilon$ carries a Reeb orbit $v_\epsilon:\R/T\Z\to\Sigma_\epsilon$ with period $T$ bounded by the dis\-place\-ment energy of $\Sigma_\epsilon$, i.e.
\beq
T\leq e(\Sigma_\epsilon)\;.
\eeq
\end{Thm}

\begin{Rmk}
In fact, Schlenk proves a much more general existence result for closed characteristics $v$ on displaceable hypersurfaces with bounds on the symplectic area enclosed by the closed characteristic. Since $\Sigma_\epsilon$ is of restricted contact type this translates into 
\beq
T=\int_0^T v^*\lambda\leq e(\Sigma_\epsilon)\;.
\eeq
\end{Rmk}

We recall that if $v_\epsilon$ is a Reeb orbit of period $T$ then the corresponding orbit of $X_\epsilon$ has period $\tau_\epsilon$ satisfying
\beq
\Lambda(E)\tau_\epsilon\leq T
\eeq 
where $\Lambda(E)$ is the constant from Proposition \ref{prop:uniform_contact_type_estimate}. Combining this with Lemma \ref{lem:bound_on_displacement_energy} we obtain the following lemma.

\begin{Lemma}\label{lemma:Ham_period_vs_Reeb_period}
The Hamiltonian vector field $X_\epsilon$ on $\Sigma_\epsilon$ has a periodic orbit of period $\tau_\epsilon$ satisfying
\beq
\textstyle
\Lambda(E)\tau_\epsilon\leq e(\Sigma_\epsilon)\leq2(2E-2\min_{\overline{\Omega}} V)^{1/2}\cdot\mathrm{diam}(\Omega)\;,
\eeq 
and thus
\bea \label{e:bound_on_period_tau}
\tau_\epsilon&\leq\frac{2(2E-2\min_{\overline{\Omega}} V)^{1/2}\cdot\mathrm{diam}(\Omega)}{\Lambda(E)}\;.
\eea
\end{Lemma}

This, of course, immediately implies that the Euler-Lagrangian equation corresponding to $L_\epsilon$ has a solution $\gamma_\epsilon$ of energy $E_\epsilon(\gamma_\epsilon)=E$  with period $\tau_\epsilon$ satisfying~\eqref{e:bound_on_period_tau}. 

For later purposes we need the additional information that the Morse index of $\gamma_\epsilon$ is bounded by $\tfrac12\dim\Sigma_\epsilon+1=N+1$. It is a classical fact that under Legendre duality between $\L^E_\epsilon$ and $K_\epsilon$ the Morse index and the Conley-Zehnder index agree. More precisely, if $(\Gamma,\tau)\in\Crit(\L^E_\epsilon)$ and the Reeb orbit $v$ correspond under Legendre duality then
\beq
\Morse\big(\Gamma;\L^E_\epsilon|_{H^1\times\{\tau\}}\big)=\CZ(v)\;.
\eeq
This identity has been proved by Viterbo in \cite{Viterbo_A_new_obstruction_to_embedding_Lagrangian_tori}, who extended a previous related result by Duistermaat \cite{Duistermaat_On_the_Morse_index_in_variational_calculus} (see also~\cite{Long_An_Indexing_domains_of_instability_for_Hamiltonian_systems,Abbondandolo_On_the_Morse_index_of_Lagrangian_systems} for alternative proofs).  

Now, let us first assume that the functional $\L^E_\epsilon$ is Morse-Bott. Via Legendre duality this translates, in the Hamiltonian formulation, to the fact that on $\Sigma_\epsilon$ is non-degenerate, i.e.~all Reeb orbits are isolated and non-degenerate, that is, the linearized \Poincare return map along a Reeb orbit has only one eigenvalue equal to 1 (which is necessarily there due to the autonomous character of the Reeb flow.) Then the proof of Theorem 4.9 in \cite{Cieliebak_Frauenfelder_Paternain_Symplectic_Topology} can be improved to show that Conley-Zehnder index of the Reeb orbit $v_\epsilon$ satisfies
\beq
\CZ(v_\epsilon)\in\{N,N+1\}\;.
\eeq
In more detail, it is shown in \cite{Cieliebak_Frauenfelder_Paternain_Symplectic_Topology} that a certain moduli space would be compact if the Reeb orbit $v_\epsilon$ did not exist. This leads then to a contradiction. Assuming that $\Sigma_\epsilon$ is non-degenerate a closer inspection of the proof shows that a gradient flow line (in the sense of Floer) of the Rabinowitz action functional connecting the orbit $v_\epsilon$ and a maximum of an auxiliary Morse function on $\Sigma_\epsilon$ has to exists. Using the index formula in \cite[Proposition 4.1]{Cieliebak_Frauenfelder_Restrictions_to_displaceable_exact_contact_embeddings} and the $\mu$-grading for Morse-Bott homology \cite[Appendix A]{Cieliebak_Frauenfelder_Restrictions_to_displaceable_exact_contact_embeddings} (see also the paragraph below equation (66) therein) this translates to
\bea
1&=\mu(v_\epsilon)-\mu(\max)\\
&=\CZ(v_\epsilon)+\eta(v_\epsilon)-\tfrac12-\underbrace{\CZ(\max)}_{=0}-\tfrac12(2N-1)\\
\eea
where $\eta(v_\epsilon)\in\{0,1\}$. This summand  is due to the fact that a critical point on the critical manifold represented by the periodic orbit $v_\epsilon$ has Morse index $0$ or $1$. The conventions for the Conley-Zehnder index in \cite{Cieliebak_Frauenfelder_Restrictions_to_displaceable_exact_contact_embeddings} agree with the ones here, see \cite[Equation (60)]{Cieliebak_Frauenfelder_Restrictions_to_displaceable_exact_contact_embeddings}. Therefore, we conclude
\bea
\CZ(v_\epsilon)\in\{N,N+1\}\;.
\eea

Thus, we conclude that  $\gamma_\epsilon=\pi(v_\epsilon)$ has Morse index $N$ or $N+1$ under the assumption that $\L^E_\epsilon$ is Morse-Bott.  

If $\L^E_\epsilon$ is degenerate we choose a sequence of compactly supported $C^\infty$-small perturbations $f_n:T^*S^N\to\R$ such that the  action functional $\L_\epsilon^{E,f_n}$ corresponding to the Lagrangian  $L_\epsilon+f_n+E$ is Morse-Bott, we find by our previous discussion a sequence $v_\epsilon^n$ of critical points of $\L^{E,f_n}_\epsilon$ such that all $v_\epsilon^n$ have period uniformly bounded from above by $e(\Sigma_\epsilon)+\delta$ for some small $\delta>0$, energy $E$, and Morse index $N$ or $N+1$. Since $f_n$ is $C^\infty$-small and the period of $v_\epsilon^n$ is uniformly bounded (see Lemma \ref{lemma:Ham_period_vs_Reeb_period}) the sequence $(v_\epsilon^n)$ converges and thus, we obtain a critical point $\gamma_\epsilon:\R/T\Z\to\Omega$ of $\L^E_\epsilon$ with 
\beq
\Lambda(E)\tau_\epsilon\leq e(\Sigma_\epsilon)+\delta,\quad E_\epsilon(\gamma_\epsilon)=E, \quad \Morse(\gamma_\epsilon)\leq N+1\;.
\eeq
Moreover, we can choose $\delta$ as small as we like. Let us summarize this discussion.

\begin{Prop}\label{prop:approx_sequence}
For any $\epsilon>0$ and $E>\max_{\overline{\Omega}}V$ there exists a critical point $(\Gamma_\epsilon,\tau_\epsilon)$ of $\L^E_\epsilon$ with 
\bea
&\tau_\epsilon\leq\frac{2(2E-2\min_{\overline{\Omega}} V)^{1/2}\cdot\mathrm{diam}(\Omega)}{\Lambda(E)},\\[2ex]
&E_\epsilon\big(\Gamma_\epsilon(\tfrac{t}{\tau_\epsilon})\big)=E,\\[2ex]
&\Morse\bigl(\Gamma_\epsilon;\L^E_\epsilon|_{H^1\times\{\tau_\epsilon\}}\bigr)\leq N+1\;.
\eea
\end{Prop}

We now have all the ingredients to prove Theorem~\ref{thm:main} and Corollary~\ref{cor:main}.
\begin{proof}[\textsc{Proof of Theorem~\ref{thm:main}}]
We fix an energy value $E>\max_{\overline\Omega} V$ and consider the sequence $\{(\Gamma_\epsilon,\tau_\epsilon)\}$ given in Proposition~\ref{prop:approx_sequence}. 

We first show that the sequence $\{\tau_\epsilon\}$ is uniformly bounded from below by some constant $T_1>0$. Indeed, assume by contradiction that $\tau_\epsilon\to0$ up to a subsequence for $\epsilon\to0$. Then, up to taking a further subsequence, by Proposition~\ref{prop:Benci_improved_without_lower_period_bound} we infer that $\Gamma_\epsilon$ converges uniformly to a constant curve $\gamma\equiv q$ with $E(\gamma)=V(q)=E$ and such that $q$ is either a critical point of $V$ or $q\in\p\Omega$ and $\nabla V(q)=-a\nu(q)$ for some $a>0$. This contradicts the assumption $E>\max_{\overline\Omega} V$.

Hence, we have
\beq
0< T_1\leq \tau_\epsilon\leq T_2:=\frac{2(2E-2\min_{\overline{\Omega}} V)^{1/2}\cdot\mathrm{diam}(\Omega)}{\Lambda(E)}\,.
\eeq
By Proposition~\ref{prop:Benci_improved}, up to taking a further subsequence for $\epsilon\to0$, $(\Gamma_\epsilon,\tau_\epsilon)$ converges to some $(\Gamma,\tau)$ in $H^1(S^1;\R^N)\times\R_{>0}$, where $T_1\leq\tau_\epsilon\leq T_2$. Let $\mu$ be the measure given by Proposition~\ref{prop:Benci_improved}. By Proposition~\ref{prop:bounde_Morse_index} and by the uniform bound on the Morse index of $\Gamma_\epsilon$, the support of $\mu$ contains at most $N+1$ points. Therefore, by Proposition~\ref{prop:Benci_improved}, the $\tau$-periodic curve $\gamma(t):=\Gamma(\tfrac t\tau)$  is a $\tau$ periodic bounce orbit of the Lagrangian system given by $L$ with energy $E(\gamma)=E$ and at most  $N+1$ bounce points.
\end{proof}

\begin{proof}[\textsc{Proof of Corollary~\ref{cor:main}}]
If the potential $V$ is constant, say $V\equiv c$, then the solutions of the Euler-Lagrange equation of $L$ with energy $E>c$ are straight curves with constant positive velocity, and therefore each of them will eventually bounce on $\partial\Omega$. Let us now consider the nontrivial case in which
\beq\label{e:corollary_nontrivial}
\textstyle \max_{\overline\Omega}|\nabla V|>0,
\eeq
and let $\gamma$ be a periodic bounce orbit with energy $E(\gamma)=E>\max_{\overline\Omega}V$ 
and no bounce points. Then $\gamma$ is a smooth, periodic solution of the Euler-Lagrange equation
\beq
\gamma''+\nabla V(\gamma)=0\,.
\eeq
Since $\gamma(t)\in\overline\Omega$ for each $t\in\R$, we can estimate
\bea\label{e:gamma_diam}
\mathrm{diam}(\overline\Omega)
&\geq
|\gamma(t)-\gamma(0)|\\
&=
\biggl|
\int_0^t
\gamma'(s)
\;
ds
\biggr|\\
&
\geq
|\gamma'(0)|\cdot|t|
-
\biggl|
\int_0^t
\!
\int_0^s
\gamma''(r)
\;dr\;ds
\biggr|\\
&
\geq
|\gamma'(0)|\cdot|t|
-
\int_0^t
\!
\int_0^s
|\nabla V(\gamma(r))|
\;dr\;ds\\
&
=
2(E-V(\Gamma(0))) |t|   
-
\int_0^t
\!
\int_0^s
|\nabla V(\gamma(r))|
\;dr\;ds\\
&
\geq
\textstyle 2(E-\max_{\overline\Omega}V)^{1/2} |t|   - \tfrac12(\max_{\overline\Omega}|\nabla V|  )  t^2\\
&
\geq
\textstyle 2(E-\max_{\overline\Omega}V)^{1/2} t   - \tfrac12(\max_{\overline\Omega}|\nabla V|  )  t^2.\\
\eea
By~\eqref{e:corollary_nontrivial} the above is possible only if
\beq
\bigl[2(\textstyle E-\max_{\overline\Omega}V)^{1/2}\bigr]^2
-4 \mathrm{diam}(\overline\Omega)\tfrac12\textstyle\max_{\overline\Omega}|\nabla V|\leq0,
\eeq
which can be rewritten as
\beq
E\leq 
\textstyle\max_{\overline\Omega}V
+
\tfrac12\mathrm{diam}(\overline\Omega)\textstyle\max_{\overline\Omega}|\nabla V|.
\eeq
This implies that all periodic bounce orbits with energy 
$E>\max_{\overline\Omega}V
+
\tfrac12\mathrm{diam}(\overline\Omega)\textstyle\max_{\overline\Omega}|\nabla V|$ have at least one bounce point.
\end{proof}

%
%
\bibliographystyle{amsalpha}
\bibliography{../../../Bibtex/bibtex_paper_list}
\end{document}